\documentclass[a4paper,13pt]{amsart}

\usepackage{amssymb,amsbsy,amsmath,amsfonts,amssymb,amscd}
\usepackage{latexsym}
\usepackage{graphics}
\usepackage{color}
\input xy
\xyoption{all}
\newtheorem{theorem}{Theorem}[section]
\newtheorem{theo}[theorem]{Theorem}
\newtheorem{remark}[theorem]{Remark}
\newenvironment{rem}{\begin{remark}\rm}{\end{remark}}

\newtheorem{defin}[theorem]{Definition}
\newenvironment{definition}{\begin{defin}\rm}{\end{defin}}

\newtheorem{question}[theorem]{Question}
\newtheorem{prop}[theorem]{Proposition}
\newtheorem{cor}[theorem]{Corollary}
\newtheorem{lemma}[theorem]{Lemma}
\newtheorem{example}[theorem]{Example}

\newcommand{\sR}{{\mathcal R}}

\newcommand{\PP}{\ensuremath{\mathbb{P}}}

\newcommand{\ZZ}{\ensuremath{\mathbb{Z}}}
\newcommand{\QQ}{\ensuremath{\mathbb{Q}}}

\newcommand{\NN}{\ensuremath{\mathbb{N}}}
\newcommand{\hol}{\ensuremath{\mathcal{O}}}

\newcommand\e{\epsilon}

\DeclareMathOperator{\Pic}{Pic}

\newcommand{\ra}{\ensuremath{\rightarrow}}

\def\eea{\end{eqnarray*}}
\def\bea{\begin{eqnarray*}}

\newcommand\dual{\mathrel{\raise3pt\hbox{$\underline{\mathrm{\thinspace d
\thinspace}}$}}}
\newcommand\qe{\ifhmode\unskip\nobreak\fi\quad $\Box$}       % box for QED

\def\BOX{\hfill\lower.5\baselineskip\hbox{$\Box$}}
\usepackage{hyperref}
\newcommand{\rk}{{\rm rk}}

\title [The $P_{12}$-theorem in char = $ p > 0$] {Enriques' classification in characteristic $ p >0$ : the $P_{12}$-Theorem }

\author{Fabrizio Catanese - Binru Li}
\address {Lehrstuhl Mathematik VIII\\
Mathematisches Institut der Universit\"at Bayreuth\\
NW II,  Universit\"atsstr. 30\\
95447 Bayreuth}
\email{Fabrizio.Catanese@uni-bayreuth.de, Binru.Li@uni-bayreuth.de}

\thanks{
properly elliptic surfaces, quasi-elliptic surfaces,
Kodaira dimension.\\
The present work took place in the framework  of the 
 ERC Advanced grant n. 340258, `TADMICAMT' }

\date{\today}

\begin{document}

\maketitle

{\em  Dedicated to David Mumford on the occasion of his 80th birthday}

\begin{abstract}
The main goal of  this paper is to show that Castelnuovo- Enriques'   $P_{12}$-theorem (a precise version of the rough classification of
	algebraic surfaces) also holds
	 for algebraic surfaces $S$ defined over an algebraically closed field $k$ of  positive characteristic ($char(k) = p > 0$).
	 
	The result relies on a main theorem  describing   the growth of the plurigenera for properly-elliptic 
	or properly quasi-elliptic surfaces (surfaces with Kodaira dimension equal to 1). We also discuss the limit cases, i.e. the families of   surfaces which 
	show that  the results 
	of the main theorem are sharp. 	
\end{abstract}

\tableofcontents

\section*{Introduction}

The main technical result of the present article,  expressed in modern language,  is the following one:

{\bf Main Theorem.}
{\em Let $S$ be a projective surface of Kodaira dimension 1 defined over an algebraically closed field $k$,
and let $K_S$ be a canonical divisor on $S$, so that  $\Omega^2_S \cong \hol_S (K_S)$.

Then the growth of the {\bf plurigenera}
$ P_n(S) = dim H^0 (\hol_S ( n K_S)) = dim H^0 ( ( \Omega^2_S )^{\otimes n})$ satisfies:

\begin{enumerate}
\item
$P_{12}(S)\geq 2$
\item
there exists $n \leq 4$ such that $P_n (S) \geq 1 $
\item
there exists $n \leq 8$ such that $P_n (S) \geq 2$
\item
$\forall n \geq 14$  $P_n (S) \geq 2$.

\end{enumerate}}

While (2)-(3) of  the result are new also in the classical case where $k$ is a field of characteristic zero, (1) is due to Enriques 
\cite{Enr14} in characteristic 0 and  (4) was shown by Katsura and Ueno \cite{KU85} for elliptic surfaces in
all characteristics (but we reprove their result here as part of the above more general statement). Needless to say, we use in the proof of our theorem   many  results,
lemmas and propositions previously established by many authors,
especially Bombieri and Mumford,
Raynaud, and Katsura-Ueno (\cite{mum1}, \cite{BM77},\cite{bm3},  \cite{Ray76}, \cite{KU85}). 

  Most important is statement (1), which  allows to extend to positive characteristic
the main classification  theorem  of Castelnuovo and Enriques.  In modern language (see the next section for more details, and a more precise 
and informative statement)   the classification theorem  says:

{\bf $P_{12}$-Theorem.}
{\em Let $S$ be a projective surface  defined over an algebraically closed field $k$.

Then for the Kodaira dimension $ Kod(S)$ we have:
\begin{itemize}
\item
(I) $ Kod (S) = - \infty \iff P_{12}(S)= 0$
\item
(II) $ Kod (S) =  0  \iff P_{12}(S)= 1$
\item
(III) $ Kod (S) =  1  \iff P_{12}(S)  \geq 2$ and, for  $S$  minimal, $K^2_S = 0$
\item
(IV) $ Kod (S) =  2  \iff P_{12}(S)  \geq 2$ and, for    $S$  minimal, $K^2_S >0$.

\end{itemize}}

It should be observed that the estimates for the growth of the plurigenera are much weaker
if one considers properly elliptic non algebraic surfaces, see \cite{iitaka} who proved the analogue of (4) of the main theorem for non algebraic surfaces. Iitaka showed  that, for  $n \geq 86$,
$H^0 (\hol_S ( n K_S))$ yields the canonical elliptic fibration. One of the reasons why the estimate is much weaker
depends on the failure of the Poincar\'e reducibility theorem, implying in the algebraic case that a certain monodromy group $G$
is Abelian. Hence, for instance, if $G$ is Abelian, it cannot be a Hurwitz group, i.e. $G$ cannot  have generators $a,b,c$ of respective orders $(2,3,7)$  satistying $abc=1$. 

Indeed (we omit here the simple proof) the analogue of statement (1) for non algebraic surfaces is that $P_{42} \geq 2$.

 Concerning higher dimensional algebraic varieties, a natural question emerges:

\begin{question}
Given a projective manifold of $X$ dimension $N$, is there a  sharp number $d = d(N)$ such that 
\begin{enumerate}
\item
$ Kod (X) = - \infty \iff P_{d}(X)= 0$
\item
$ Kod (X) =  0  \iff P_{d}(X)= 1$
\item
$ Kod (X)  \geq 1   \iff P_{d}(X) \geq 2$ ?
\end{enumerate}

\end{question}

Progress on a related  question, about effectivity of the Iitaka fibration, was obtained, among others,   by  Fujino and Mori \cite{fujino-mori} and Birkar and Zhang \cite{birkar-zhang}.

\section{The classification theorem of Castelnuovo and Enriques}
Let $S$ be a nonsingular projective  surface defined over an algebraically closed field $k$, and let $K_S$ be a canonical divisor on $S$,
so that $\Omega^2_S \cong \hol_S (K_S)$. We assume that $S$ is minimal: this means that  there does not exist an irreducible curve $C$ of the fist kind,
i.e., a curve $C$ with $C^2 = K_S \cdot C= -1$. Let us recall the definition of the basic numerical invariants associated to $S$, which
allow its birational classification.

For each  integer $m \in \NN$, we denote as usual, following Castelnuovo and Enriques, by  $$P_m (S):=h^0(S,mK_S),$$ the {\bf $m$-th plurigenus} of $S$. 

In particular, the {\bf geometric genus is}  $p_g (S) : = P_1 (S)$, while the {\bf arithmetic irregularity} is defined as $ h (S)  : = h^1 (\hol_S)$,
and the  {\bf arithmetic genus is} defined as $$p_a (S) : = p_g (S)  - h(S)  = \chi (\hol_S) - 1.$$
To finish our comparison of classical and modern notation, recall that the {\bf geometric irregularity} is defined as $ q (S)  : = \frac{1}{2}b_1(S)$,
where $b_1(S) $ is the first l-adic Betti number of $S$,  $b_1(S) : = dim_{\QQ_l} H^1_{e t} (S, \QQ_l).$

$q(S) $ is equal to the dimension of the Picard scheme $\Pic^0(S)$, and also of the dual scheme $\Pic^0(\Pic^0(S))$, 
and of the Albanese variety  $Alb(S): = \Pic^0(\Pic^0(S)_{red})$.

The above numbers are all equal in characteristic zero: $q(S) = h(S) = h^0(\Omega^1_S)$, but not in characteristic $p >0$,
where one just has some inequalities.

Since $H^1 (\hol_S)$ is the Zariski tangent space to  the Picard scheme at the origin (\cite{orangebook}),
one has the inequalities (cf. \cite{BM77})
$$  h (S) \geq  q(S),  \  2 p_g (S) \geq \Delta : = 2 ( h(S) - q(S)) =  2  h(S) - b_1(S) \geq 0.$$

The inequality $ h^0(\Omega^1_S) \geq q$ was shown by Igusa \cite{Igu1},  and there are examples 
where equality does not hold, cf.
\cite{Igu2}, \cite{mumpath} \footnote{The space of regular one forms on the Albanese variety $A$
pulls back injectively to a subspace $V$ of the space $H^0 (\Omega^1_S)$,
contained in the space of $d$-closed forms; it is an open question how to characterize $V$, for instance Illusie suggested 
$V$ could be  the intersection of the  kernels of $ d \circ C^{m}$,
where $C$ is the Cartier operator, and $m$ is any positive integer: see \cite{seshadri}, \cite{DI}, \cite{Ill}}.

Moreover, the {\bf  linear genus } $ p^{(1)} (S): = K_S^2  + 1$ is the arithmetic genus of any canonical divisor on the minimal surface. It is a birational invariant for
every non ruled algebraic  surface.

The classification of smooth projective curves $C$  is given in terms of the genus $g(C) : = h^0(\hol_C(K_C))$,  
\begin{itemize}
	\item[(I)] $g(C) = 0  \iff C \cong \PP^1$.
	\item[(II)] $g(C) = 1  \iff   \hol_C \cong \hol_C (K_C)   \iff  C$ \ {\it  is an elliptic curve (it is  isomorphic to a plane cubic curve)}.
	\item[(III)] $g(C) \geq 2   \iff  C$ \ {\it   is of general type, i.e., $H^0(C, \hol_C (m K_C))$ yields an embedding of $C$
	for all $m \geq 3$.}
	\end{itemize}

Enriques and Castelnuovo (\cite{Enr14}  and \cite{CE}) were able to give the surface classification essentially  in terms of $P_{12}(S)$,
as follows:
\begin{theo} {\bf ( $P_{12}$-theorem of Castelnuovo-Enriques)}

Let $S$ be a projective smooth  surface defined over an algebraically closed field $k$ of characteristic zero,
and let  $ p^{(1)} (S): = K_S^2  + 1$ be the linear genus of a minimal model in the birational equivalence class of $S$. Then 
\begin{itemize}
	\item[(I)] $ P_{12}(S)=0 \iff $ {\it $S$ is ruled} $\iff $ {\it $S$ is birational to a product $ C \times \PP^1 $,  $g(C) = q(S) = h(S)$}.
	\item[(II)] $P_{12}=1  \iff  \hol_S \cong \hol_S ( 12 K_S)   $.
	\item[(III)] $ P_{12}\geq 2$ and $p^{(1)} (S)=1  \iff ${\it $S$ is properly elliptic, i.e.  $H^0(S, \hol_S ( 12 K_S))$  yields a fibration over a curve with
	general  fibres elliptic curves.}
	\item[(IV)] $P_{12}\geq 2$ and $p^{(1)} (S) > 1 \iff $ {\it $S$ is of general type, i.e. $H^0(S, \hol_S ( m K_S))$ yields a birational  embedding of $S$
	for $m $ large ($m \geq 5$ indeed suffices, as conjectured by Enriques in \cite{EnrSup} and proven by Bombieri \cite{Bom73})}.
\end{itemize}

Moreover, if $S$ is minimal, then in modern terminology:

\begin{itemize}
\item
{\bf Case (I):} $S \cong \PP^2$ or $S$ is a $\PP^1$-bundle over a curve $C$,
\item
{\bf Case (II), $p_g(S) = 1 , q(S) = 2 \iff   \hol_S \cong \hol_S (  K_S),  q(S) = 2 \iff $} $S$ is an Abelian surface .
\item
{\bf Case (II), $p_g(S) = 1 , q(S) = 0  \iff   \hol_S \cong \hol_S (  K_S),  q(S) = 0 \iff$} $S$ is a  K3 surface.
\item
{\bf Case (II), $p_g(S) = 0 , q(S) = 0  \iff   \hol_S \ncong \hol_S (   K_S), \  \hol_S \cong \hol_S (  2 K_S),  q(S) = 0 \iff$} $S$ is an Enriques  surface.
\item
{\bf Case (II), $ q(S) = 1 ( \Rightarrow p_g(S) = 0 ) \iff  \hol_S \ncong \hol_S (   K_S), \  \hol_S \cong \hol_S (  m K_S), \ m \in \{2,3,4,6\},  q(S) = 1 \iff $} $S$ is a hyperelliptic surface.
\item
{\bf Case (III), $ p_a(S) = - 1 \iff  S \cong C \times E,  \ g(E) = 1, \ g(C) = q(S) -1 .$}
\end{itemize}

\end{theo}

A modern account of the Castelnuovo-Enriques classification of surfaces was first given in \cite{shaf} and in \cite{kodclass},
then it appeared also in  \cite{bh}, \cite{bea} (this is the only text which mentions
the $P_{12}$-theorem, in the historical note on page 118), later also in  \cite{badescu} and \cite{bpv}.

\begin{rem}

i) Nowadays, cases (I)-(IV) are distinguished according to the Kodaira dimension, which is defined to be 
$- \infty$ if all the plurigenera vanish
 ($P_n = 0 \  \forall n \geq 1$),
otherwise it is defined as the maximal dimension of the image of some $n$-pluricanonical map (the map associated to $H^0(\hol_X(nK_X))$).

ii) The occurrence of the number $12$ is rather miracolous: it first appears since, by the canonical divisor formula \ref{bm}, 
in case (II)  the equation $$2 = \sum_j (1 - \frac{1}{m_j})  $$
admits only the following (positive)  integer solutions:
$$ (2,2,2,2) ,(3,3,3), (2,4,4), (2,3,6) $$
and then we get a set of integers $m_j$ whose least common multiple is precisely $12$.

According to the several cases we have $ 2K_S \equiv 0, 3 K_S \equiv 0, 4 K_S \equiv 0, 6 K_S \equiv 0, $
where $D \equiv 0$ means that that $F$ is linearly equivalent to zero,
i.e. $ \hol_S(D) \cong \hol_S$. It follows that in case (II) we have $ 12K_S \equiv 0$, hence $P_{12}=1$.

The second occurrence  is more subtle, and is the heart of the theorem:  in case (III) one has $P_{12} \geq 2$.
\end{rem}

It is now customary (the name 'key theorem' is due to  \cite{bea})  to see the two major steps of surface classification as follows:

\begin{theo} {\bf ( Key Theorem)}
If $S$ is minimal, then

$K_S$ is {\bf nef} (i.e. , $ K_S \cdot C \geq 0$ for all curves $ C \subset S$)
$\iff$ $S$ is nonruled.
\end{theo}

\begin{theo} {\bf ( Crucial  Theorem)}\label{crucial}
 $S$ is minimal, with $p_g (S) =0, q(S) =1$

$\iff$ $S$ is {\bf isogenous} to an elliptic  product, i.e. $S$ is the quotient  $(C_1 \times C_2)/G$ of a product of  curves
of genera $$g_1: = g (C_1) =1 , g_2: = g (C_2) \geq 1,$$
 by a free action of a finite group of product type
($G$ acts faithfully on $C_1, C_2$ and we take the diagonal action $ g (x,y) : = ( g x, gy)$),
and such that moreover  if we denote by $g'_j = g (C_j /G)$, then $g'_1 + g'_2 = 1$.

More precisely, let $A$ be the Albanese variety of $S$, which is an elliptic curve and let 
$$ \alpha : S \ra A$$
be the Albanese map.

Then either:

1) $S$ is a hyperelliptic surface, $(C_1 \times C_2)/G$, $g_2 = 1$,
 $G$ is a subgroup of translations of $C_1$, $ A = C_1 / G$,
 while $C_2 / G \cong \PP^1$.
 
 In this case all the fibres of the Albanese map are isomorphic to $C_2$, $P_{12}(S) = 1$ and  $S$ admits also an elliptic fibration
 $\psi : S \ra C_2 / G \cong \PP^1$.
 
 2)  $S$ is properly elliptic ($P_{12}(S) \geq  2$) and the genus $g = g_2 $ of the Albanese fibres satisfies $g_2 \geq 2$: again
 $G$ is a subgroup of translations of $C_1$, $ A = C_1 / G$,
 $C_2 / G \cong \PP^1$,  all the fibres of the Albanese map are isomorphic to $C_2$.
 
 3) $S$ is properly elliptic ($P_{12}(S) \geq  2$) and the genus $g = g_1  $ of the Albanese fibre satisfies $g_1=1$:
 $ A = C_2 / G$, $C_1 / G \cong \PP^1$, and the fibres of the Albanese map
 $$ \alpha : S = (C_1 \times C_2)/G  \ra A = C_2 / G $$
 are either isomorphic to  the elliptic curve $C_1$ or are multiples of a smooth elliptic curve isogenous to $C_1$.
 
 \begin{rem}\label{mon}
 A crucial observation, used by Enriques in \cite{Enr14} for the $P_{12}$-theorem  is that in the first two cases the group $G$ is Abelian.
 The crucial ingredient is the canonical divisor (canonical bundle) formula, established by Enriques, Kodaira, and then extended to positive characteristic 
 by Bombieri and  Mumford.

 \end{rem}
 
 \begin{theo}\cite[p.27 Theorem 2.]{BM77}\label{bm}
	Let $f:S\rightarrow C$ be a relatively minimal fibration such that the arithmetic genus of a fibre equals $1$ (the general fibre is necessarily
	smooth elliptic in characteristic zero, but it can be  rational with one cusp in characteristic $2,3$: the latter  is called the quasi-elliptic case).

	 Let $\{q_1,...q_r\}\subset C$ the set of points over which the fibre $f^{-1}(q_i)=m_iF'_i$ is a multiple fibre (i.e. $m_i\geq 2$ and $F'_i$ is not a multiple of any 
	 proper sub-divisor), and consider
	 the  coherent sheaf $R^1f_*(\mathcal{O_S})$ on the smooth curve $C$, which decomposes as
$$R^1f_*(\mathcal{O}_S)= \hol_C(L) \oplus T,$$
where $\hol_C(L) $ is an invertible sheaf and $T$ is a torsion subsheaf with $ supp (T) \subset  \{q_1,...q_r\} $.
The fibres over the points of $ supp (T)$ are called {\bf wild fibres}, moreover 
 $T=0$ if $char (k) = 0$.

	 Then
	$$ K _S=f^*( \delta ) + \sum_{i=1}^{r}a_iF'_i, \delta :  = - L +  K_C$$
	where
	\begin{itemize}
		\item[(i)] $0\leq a_i<m_i$;
		\item[(ii)] $a_i=m_i-1$ if $m_iF'_i$ is not wild (i.e.,  $q_i \notin supp (T)$);
		\item[(iii)] $d : = \deg \ (\delta) = \deg  \ (- L +  K_C) =2g(C)-2+\chi(\mathcal{O}_S)+length (T)$,
		
		 where $g(C)$ is the genus of $C$. 
	\end{itemize}
\end{theo}

\end{theo}

Let us see how the above applies  in characteristic zero and in  the special
subcase: $p_g=0, q=1$,    genus of the Albanese fibres equal $1$, there exist multiple fibres.  

Then,  for $n =2$, since  we have $ \deg (\delta) = 0$, follows that 
$$ 2 K_S =   \sum_{i=1}^{r} (m_i - 2 )F'_i + f^*( \delta  +  \sum_{i=1}^{r} q_i).$$ 
The divisor $ \delta  +  \sum_{i=1}^{r} q_i$ is effective by the Riemann Roch theorem on the elliptic curve $A$,
so we have written $2 K_S$ as the sum of two effective divisors.

 Hence we obtain that $P_2 \geq 1$,
and similarly  one gets that $P_{12} \geq 6$.

\section{The $P_{12}$-theorem in positive characteristic}

The extension of the Castelnuovo-Enriques  classification of surfaces to the case of  positive characteristic was achieved  by D. Mumford and E. Bombieri (cf. \cite[Section 3]{bm3}, \cite[Theorem 1.]{BM77}).

In a remarkable series of three papers   they  got  most of the following  full result.

\begin{theo} {\bf ( $P_{12}$-theorem in positive characteristic )}\label{12}

Let $S$ be a projective smooth  surface defined over an algebraically closed field $k$ of characteristic $p >0$,
and let  $ p^{(1)} (S): = K_S^2  + 1$ be the linear genus of a minimal model in the birational equivalence class of $S$. Then 
\begin{itemize}
	\item[(I)] $ P_{12}(S)=0 \iff $ {\it $S$ is ruled} $\iff $ {\it $S$ is birational to a product $ C \times \PP^1 $,  $g(C) = q(S) = h(S)$}.
	\item[(II)] $P_{12}=1  \iff  \hol_S \cong \hol_S ( 12 K_S)   $.
	\item[(III)] $ P_{12}\geq 2$ and $p^{(1)} (S)=1  \iff ${\it $S$ is properly elliptic or properly quasi-elliptic, i.e.  $H^0(S, \hol_S ( 12 K_S))$  yields a fibration over a curve with
	general  fibres either elliptic curves or rational curves with one cusp.}
	\item[(IV)] $P_{12}\geq 2$ and $p^{(1)} (S) > 1 \iff $ {\it $S$ is of general type, i.e. $H^0(S, \hol_S ( m K_S))$ yields a birational  embedding of $S$
	for $m $ large (indeed, $ m \geq 5$ suffices).}
\end{itemize}

\end{theo}

Moreover, Bombieri and Mumford  in \cite{BM77} and \cite{bm3} gave a full description of the surfaces in the classes (I) and (II) (with new non classical surfaces), but 
classes (II) and (III) were not distinguished by the behaviour of the 12-th plurigenus, but only by 
the Kodaira dimension, i.e., by the growth of $P_n(S)$ as $ n \ra \infty$.

The sharp statement ($ \forall m \geq 5$) in case $(IV)$, established by Bombieri  \cite[Main Theorem]{Bom73} in characteristic zero,
was extended by  T. Ekedahl's  to the case of  positive characteristic (cf. \cite[Main Theorem]{Eke88}, see also \cite{cf}
and \cite{4authors} for a somewhat simpler proof.

\section{Auxiliary results and proof of the $P_{12}$-theorem}

Case $(III)$ can be divided into two subcases: properly elliptic fibrations and properly quasi-elliptic fibrations.

Recall the definition of quasi-elliptic surfaces:

\begin{defin}
  A {\em quasi-elliptic} surface $S$ is a nonsingular projective surface admitting a fibration $f:S\rightarrow C$ over a nonsingular projective curve $C$ such that $f_*\mathcal{O}_X=\mathcal{O}_C$ and such that the general fibres of $f$ are rational curves with one cusp.\\
  If the fibration $f$ is induced by $H^0(S, \hol_S (n K_S))$ for some $ n > 0$, we call $S$ a {\em properly quasi-elliptic surface}.
\end{defin}
\begin{remark}
	1) By a result of J. Tate (cf. \cite{Tat52}), quasi-elliptic fibrations only appear in characteristic 2 and 3.
	
	2) in case (III), where $P_n (S) : = dim H^0(S, \hol_S (n K_S))$ grows linearly with $n$, $S$ is necessarily properly elliptic or properly quasi-elliptic.
\end{remark}

The case  where $S$ admits a  properly elliptic fibration was treated  by T. Katsura and K. Ueno who proved in \cite[Theorem 5.2.]{KU85}  that for any properly elliptic surface $S$,  $\forall m\geq 14$, $P_m(S)\geq 2$ and showed the
existence of  an example where $P_{13}=1$. 
They show that the situation is essentially the same as in characteristic zero.    The fact  that  $P_{12}(S)\geq 2$ follows from
our more general  theorem, which uses several  auxiliary results developed by Raynaud and   Katsura-Ueno
(they  will be recalled in the sequel).

\begin{theorem}\label{sharp}{\bf (Main Theorem)}
Let $f:S\rightarrow C$ be a properly elliptic or quasi-elliptic fibration. Then

\begin{enumerate}
\item
$P_{12}(S)\geq 2$
\item
there exists $n \leq 4$ such that $P_n (S) \neq 0$
\item
there exists $n \leq 8$ such that $P_n (S) \geq 2$
\item
$\forall n \geq 14$  $P_n (S) \geq 2$.

\end{enumerate}

\end{theorem}

\begin{remark}\label{record}
Let us indicate the  examples (see remark \ref{limit}) which show that in theorem \ref{sharp} the  inequalities
in our assumptions  are the best possible.

(2) and (4): in the notation  of (2) of theorem \ref{crucial} we let  $G = \ZZ/2 \oplus \ZZ/6$;
clearly $G$ is isomorphic to a subgroup of any elliptic curve. In order to obtain a curve $C_2$ with a $G$ action such that $C_2/G \cong \PP^1$
we consider a $G$-Galois covering $C_2$ of $\PP^1$ branched in $3$ points, and with local monodromies  
$$ (1,0), (0,1) , (-1,-1) .$$ 
This exists, by Riemann's existence theorem since the sum of the three local monodromies equals zero.
This example yields the curve $C_2$ with affine equation $ y^2 = x^6 -1$, which is smooth in characteristic $\neq 2,3$,
see \cite{KU85}.

The fibration $ f : S \ra C_2/G \cong \PP^1$ is elliptic and has exactly three singular fibres, multiple with mutiplicities $2,6,6$.
It follows that 
$$ P_n (S) =  -2n + 1  + [ n/2]  +  2 \cdot [5n/6] ,$$
where $ [ a]$ denotes the integral part of $a$.

Follows that $P_1=P_2=P_3 = 0$, $P_4 = P_5 = 1$, $P_6 = 2$, $P_{13} = 1$.

(2) and (3) : in the notation  of (2) of theorem \ref{crucial} we let 
$G = \ZZ/10$;
clearly $G$ is isomorphic to a subgroup of any elliptic curve. In order to obtain a curve $C_2$ with a $G$ action such that $C_2/G \cong \PP^1$
we consider a $G$-Galois covering $C_2$ of $\PP^1$ branched in $3$ points, and with local monodromies  
$$ (5), (4) , (1) .$$ 
This exists, by Riemann's existence theorem since the sum of the three local monodromies equals zero.
Indeed, the curve is defined by the affine equation $ y^2 = x^5 -1$ and is smooth in characteristic $\neq 2,5$.

The fibration $ f : S \ra C_2/G \cong \PP^1$ is elliptic and has exactly three singular fibres, multiple with mutiplicities $2,5,10$.
It follows that 
$$ P_n (S) =  -2n + 1  + [ n/2]  +   [4n/5]  + [9n/10].$$

Follows that $P_1=P_2=P_3 = 0$, $P_4 = P_5 =P_6 = P_7 = 1$, $P_8 =P_9=  2$, $P_{10} = 3, P_{11} = 1, P_{12} = P_{13} = 2$.

\end{remark}

In  the case of properly quasi-elliptic fibrations, we shall use some result of Raynaud,  \cite{Ray76},
and a corollary developed by Katsura and Ueno (lemmas 2.3 and 2.4 \cite{KU85}).

Given a multiple fibre $mF'$ we denote by $\omega_n:=  \hol_{nF'} (K_S + n F')$ the dualizing sheaf of $nF'$. 

Observe that $F'$ is an indecomposable divisor of elliptic type,
hence (see \cite{mum1} \cite{4authors}) for any degree zero divisor $L$ on $F'$, we have 
$h^0(\hol_{F'} (L)) = h^1(\hol_{F'} (L) )$, and these dimensions are either $=0$, or $=1$,
the latter case occurring if and only if
$ \hol_{F'} (L) \cong \hol_{F'} $.

Consider now the exact sequence 
$$  0 \ra  \hol_{F'} (- (n-1)F') \ra  \hol_{nF'} \ra  \hol_{(n-1)F'} \ra 0,$$
and apply the previous remark for $ L = - (n-1) F'$ to infer that 
$$  h^0(\hol_{nF'})  =   h^0(  \hol_{(n-1)F'} ) \  {\it or} \  =   h^0(  \hol_{(n-1)F'} )+ 1,$$
 the second possibility occurring  only if 
 $$ (**) \   \hol_{F'} ( (n-1)F') \cong  \hol_{F'} .$$
 Conversely, if (**) holds, either $h^0(\hol_{nF'})  =   h^0(  \hol_{(n-1)F'} ) $ and $h^1(\hol_{nF'})  =   h^1(  \hol_{(n-1)F'} ) $,
 or both $h^0, h^1$ grow by $1$ for $nF'$.
 
 This in any case shows that the function $ h^0(\hol_{nF'})$ is monotone nondecreasing.
 One says that $n$ is a {\bf jumping value} if $ 
 n \geq 1$ and $  h^0(\hol_{nF'})  =   h^0(  \hol_{(n-1)F'} )+ 1.$ Considering all the $n \geq 1$,
 we can then define the first  jumping value, the second, and so on (they are then clearly $\geq 2$).

Recall now:

 \begin{prop}[\cite{BM77} Proposition 4 and \cite{rayIHES} Proposition 6.3.5.] \label{Prop2}
 Since $(\mathcal{O}_{F'_i} (F'_i))$ is a torsion element in the Picard group of $F'_i$, we consider its torsion order:
 $$\nu_i : =order(\mathcal{O}_{F'_i} (F'_i)).$$ We have then
	\begin{itemize}
		\item[(1)] $\nu_i$ divides both $m_i$ and $a_i+1$;
		\item[(2)] letting $ p = char (k)$,   there exists an integer $e_i \geq 1$
		such that  $m_i = \nu_i \cdot p^{e_i}$ ;
		\item[(3)] $h^0(F'_i,\mathcal{O}_{(\nu_i+1)F'_i})\geq 2$, $h^0(F_i,\mathcal{O}_{\nu_iF'_i})=1$,
		so that $\nu_i + 1$ is a jumping value;
		\item[(4)] $h^0(F'_i,\mathcal{O}_{rF'_i})$ is a non-decreasing function of  $r$.
	\end{itemize}
\end{prop}
Using \ref{Prop2}, we get the following corollary
\begin{cor}[\cite{BM77}, Corollary.] \label{Cor}
	If $h^1(S,\mathcal{O}_S)\leq 1$, we have either
	$$a_i = m_i - 1 $$ or $$a_i =m_i - 1 - \nu_i. $$
\end{cor}

More precise results are the following two lemmas of M. Raynaud (cf. \cite{Ray76}, \cite[Section 2]{BT14}).
\begin{lemma}\cite[Corollaire 3.7.6.]{Ray76} \label{ray1} Let $f:S\rightarrow C$ be an elliptic or quasi-elliptic
 fibration with $f^{-1}(q)=mF'$ a multiple fibre over $q\in C$. Then for any integer $n\geq 2$:
     \begin{itemize}
     \item[(i)] The dualizing sheaf $\omega_n : =  \hol_{nF'} (K_S + n F')$ of $nF'$ is non-trivial iff $h^0(\omega_n)=h^0(\omega_{n-1})$. 
     \item[(ii)] $\omega_n$ is trivial iff $h^0(\omega_n)=h^0(\omega_{n-1})+1$.
     \end{itemize}
\end{lemma}

\begin{lemma}\cite[Lemma 3.7.7.]{Ray76}\label{ray2} Notation being  as in Lemma \ref{ray1}, observe that the invertible sheaves 
$\mathcal{O}_{nF'}(F')$ are torsion elements in the Picard group of $nF'$. There are only two possibilities for their torsion orders.
Setting  $ o_n : = Ord(\mathcal{O}_{nF'}(F'))$ (hence $o_1 = \nu$),  we have  
	\begin{itemize}
	 \item[(i)] $o_n = o_{n-1} $;
	 \item[(ii)] $o_n = p \ o_{n-1} $.
	 \end{itemize}
	 Moreover, case $(ii)$	occurs only if $\omega_n$ is trivial.
\end{lemma}
\begin{proof}
  Setting $\mathfrak{N}:=\mathcal{O}_{F'}(-(n-1)F')$, we consider the following two exact sequences: 
   \begin{equation}\label{eq1}
   0\rightarrow \mathfrak{N} \rightarrow \mathcal{O}_{nF'}\rightarrow \mathcal{O}_{(n-1)F'}\rightarrow 0.
   	\end{equation}
 
\begin{equation}\label{eq2}
    0\rightarrow 1+\mathfrak{N}\rightarrow \mathcal{O}^*_{nF'}\rightarrow \mathcal{O}^*_{(n-1)F'}\rightarrow 0.
\end{equation}
Since $\mathfrak{N}^2=0$, the map $x\mapsto 1+x$ defines an isomorphism of abelian sheaves: $$\beta:  \mathfrak{N}\simeq 1+\mathfrak{N}.$$
Taking the induced long exact sequence of (\ref{eq1}) and (\ref{eq2}) and observing that $H^2(F',\mathfrak{N})\simeq H^2(F',1+\mathfrak{N})=0$, we get
\begin{equation}
 H^0(\mathcal{O}_{(n-1)F'})\xrightarrow{\partial}H^1(\mathfrak{N})\rightarrow H^1(\mathcal{O}_{nF'})\xrightarrow{\alpha}H^1(\mathcal{O}_{(n-1)F'})\rightarrow 0
\end{equation}
and
\begin{equation}
 H^0(\mathcal{O}^*_{(n-1)F'})\xrightarrow{\partial^*}H^1(1+\mathfrak{N})\rightarrow \Pic(nF')\xrightarrow{\alpha^*} \Pic((n-1)F')\rightarrow 0.
\end{equation}
By a result of F. Oort (cf. \cite[\S 6]{Oor62}), we have that $H^1 (\beta) (Im(\partial))=Im(\partial^*)$.\\
Since $H^1(\mathfrak{N})$ is a $\mathbb{Z}/p\mathbb{Z}$-vector space, we see that any element in $\ker(\alpha^*)$ has $p$-th power equal to $1$, hence we have $o_n=o_{n-1}$ or $o_n=po_{n-1}$.\\
If $o_n=po_{n-1}$, then $ker(\alpha^*)\neq \{1\}$ and hence $ker(\alpha)\neq \{0\}$. Since  $h^1(nF',\mathcal{O}_{nF'})=h^0(\omega_n)$, by lemma \ref{ray1} we have that $h^0(\omega_n)=h^0(\omega_{n-1})+1$ and $\omega_n$ is trivial.\\
\end{proof}

Assume that we have a multiple fibre over the point $q_j$, and
denote by $t_j$ the length of the skyscraper sheaf $T$ at $q_j$.

 Then, by the base change theorem we have
$$ t_j + 1 = \rk_{q_j}   \sR^1 f_* (\hol_S) = h^1 (\hol_{m_j  F'_j}) = h^0 (\hol_{m_j  F'_j}).$$ 

The two lemmas by Raynaud imply the following very useful corollary, which holds
more generally also for quasi-elliptic fibrations.

\begin{cor}\cite[Lemma 3.7.9]{Ray76} \cite[Lemmas 2.3, 2.4]{KU85}\label{KU}
(1) Let $n^{(i)}_j $  be the $i$-th jumping value of a wild fibre $mF'_j$ (recall that $ n^{(i)}_j\geq 2)$. 

Setting $\nu_j : =Ord(\hol_{F'_j}(F'_j))$, we have 
$$ n^{(1)}_j= \nu_j+1,$$ 

and 
$$\ n^{(2)}_j=2\nu_j+1\ {\it if}\ Ord(\hol_{(\nu_j+1)F'_j})=\nu_j, {\it or}\ =(p+1)\nu_j +1\ {\it if}\  Ord(\hol_{(\nu_j+1)F'_j})=p\nu_j. $$
(2) If $h^0(\hol_{mF'_j})=2 \Leftrightarrow t_j = 1$, then the contribution $a_j F'_j$ to the canonical divisor formula satisfies
$a_j = m_j-1$ or $ a_j= m_j -1- \nu_j$.

(3) If $h^0(\hol_{mF'_j})=3 \Leftrightarrow t_j = 2$, then 
$a_j = m_j-1$ or $ a_j= m_j -1- \nu_j$  or $ a_j = m_j-1 - 2 \nu_j$ or $ a_j = m_j-1 - (p+1)  \nu_j$.

\end{cor}

Finally, Katsura and Ueno proved for elliptic fibrations in characteristic $p$ the analogue
of a result which in characteristic zero follows from the description of the fundamental group
of the complement of a finite set of points in $\PP^1$. 

\begin{definition}[\cite{KU85}, Definition 3.1.] \label{def6}
   Let $f:S\rightarrow \mathbb{P}^1$ be an elliptic fibration with $\chi(S,\mathcal{O}_S)=0$, let $m_iF'_i$, $i=1,..,k$, be the multiple fibres, and let as usual $\nu_i$ be the torsion order of $\hol_{F'_i}(F'_i)$.
   
    Then $S$ is said to be  {\em of type $(m_1,...,m_r|\nu_1,...,\nu_r)$}.
\end{definition}

\begin{definition}[\cite{KU85}, Definition 3.2.]
 Given $1\leq i\leq r$, we say that two sequences $(m_1,...,m_r|\nu_1,...,\nu_r)$ satisfy condition $U_i$, 
 if there do exist integers $n_1,...,n_r$ (depending on $i$)  such that
 \begin{itemize}
      \item $n_i \equiv 1 \mod \nu_i$ and
      \item $\sum_{j=1}^{r} n_j/m_j \in \mathbb{Z}$.	
 \end{itemize}
 
\end{definition}

\begin{theorem}[\cite{KU85}, Theorem 3.3.]\label{Mon}
  In the situation of  \ref{def6}, then the sequences $(m_1,...,m_r|\nu_1,...,\nu_r)$ satisfy  condition $U_i$
   $ \forall i=1,2,...,r$.
\end{theorem}

\section{Proof of the main   theorem \ref{sharp}}

 Let $f:S\rightarrow C$ be a relatively minimal properly elliptic or properly quasi-elliptic fibration.
 Set here $g : = g(C)$ and  set  $t=length(T)$, where $T$ is the torsion sheaf appearing in the canonical bundle formula.
 
 The first important observation is that in the canonical bundle formula the term 
 $\chi (\hol_S)$ is $ \geq 0$, by Mumford's extension of Castelnuovo's theorem (\cite{mum1}).
 
 The case $\chi (\hol_S) + t  \geq 1, \ g \geq 1$ is quickly disposed of by observing that
 $$ P_n (S) = h^0 (\hol_S (n K_S)) \geq  h^0 (\hol_C (n \delta)) \geq g + (n-1) \geq n.$$
 
 If $ g \geq 2$, and $\chi (\hol_S) = t =0$, we are done, since then $P_n \geq (2n-1)(g-1)$.
 
 If instead $g=1, \chi (\hol_S) = t =0$ there are no wild fibres, and since the canonical divisor is
 not numerically trivial, $\sum_j ( 1 - \frac{1}{m_j}) > 0$,
 hence 
 $$ n K_S =  \sum_j n ( m_j - 1) F'_j   = \sum_j [\frac{n ( m_j - 1)}{m_j}] F_j + m_j \{ \frac{n ( m_j - 1)}{m_j} \}  F'_j,$$
 so we can rewrite
  $$ n K_S =   \sum_j f^* ([\frac{n ( m_j - 1)}{m_j}]  q_j )+ D,$$
 where $D$ is an effective divisor (with integral coefficients).
 
 Hence $$ (*) \ P_n (S) \geq   \sum_j [\frac{n ( m_j - 1)}{m_j}] \geq [ \frac{n}{2} ].$$
 
 We may therefore assume that $g=0$.
 
 For $ g=0$, if $ \chi (\hol_S) + t  \geq 3$, we get  $P_n (S) \geq n+1$.
 
 If $g=t=0,  \chi (\hol_S) = 2$, then again there are no wild fibres and the same argument as in $(*)$ yields
 $$  \ P_n (S) \geq  1 +  \sum_j [\frac{n ( m_j - 1)}{m_j}] \geq 1 +  [ \frac{n}{2} ].$$

 We are left with  the following possibilities: \\
   Case (1) $\chi(\mathcal{O}_S)=1,\ t=1$ and $g=0$;\\
     Case (2) $\chi(\mathcal{O}_S)=0,\ t=2$ and $g=0$;\\
            Case (3) $\chi(\mathcal{O}_S)=0,\ t=1$ and $g=0$;\\
        Case (4) $\chi(\mathcal{O}_S)=t=0$ and $g=0$.\\ 
        
      The next lemma shows that, except possibly in case (1), we need to take care 
      only of the properly elliptic case.

\begin{lemma}
	There exists no quasi-elliptic fibration $f:S\rightarrow \mathbb{P}^1$ with $\chi(\mathcal{O}_S)=0$.
\end{lemma}
\begin{proof}
	Assume we have such a fibration. 
	
	Let $\alpha: S\rightarrow A$ be the Albanese map of $S$ and assume that 
	$q : = dim(A)\geq 1$. Since a general fibre of $f$ is a cuspidal rational curve, whose image in $A$ must be a single point, we see that $\alpha$ factors through $f$. Hence the image of  $\alpha$
	is a point: since the image generates $A$, $A$ is a point and $q=0$,
	 a contradiction.
	 
	We conclude that  $q=0$, hence  $ p_g \geq h$ and  $ \chi (\hol_S) \geq 1$, a contradiction.
	
\end{proof}

Let us now proceed with the proof.

We can write
$$ K_S \equiv d F + \sum_i a_i F'_i,$$
where $F$ is a fibre of $f$.
We observe that $ p_g (S) = max (0, d+1)$.

Indeed, if $p_g \geq 1$,  
we can write $|K_S| = |M| + \Phi$, where $\Phi$ is the fixed part, and where the movable part
is of the form $(p_g-1) F$.

Hence $K_S$ is linearly equivalent to an effective
divisor $D$ of  the form $ K_S \equiv (p_g-1)   F + \sum_i b_i F'_i,$ with $ 0 \leq b_i < m_i$.   

If   $ d \geq 0$, then $p_g = d+1$ and the fixed part $ \Phi =  \sum_i a_i F'_i$.

Otherwise, if $d <0$ , and we assume $p_g \geq1$ we have a linear equivalence  of effective divisors:
$ (| d | + p_g-1) F + \sum_i b_i F'_i  \equiv \sum_i a_i F'_i$ which shows that $| d | + p_g-1 = 0$, 
a contradiction.

Hence in our cases we have respectively:
\\
   Case (1) $\chi(\mathcal{O}_S)=1,\ t=1, \ h=1, \ p_g=1$ and $g=0$;\\
     Case (2) $\chi(\mathcal{O}_S)=0,\ t=2, \ \ h=2, \ p_g=1$ and $g=0$;\\
            Case (3) $\chi(\mathcal{O}_S)=0,\ t=1, \ h=1, \ p_g=0$ and $g=0$;\\
        Case (4) $\chi(\mathcal{O}_S)=t=0, \ h=1, \ p_g=0$ and $g=0$.\\ 
        
        Observe therefore that corollary \ref{Cor} applies in all cases except (2).
        
       {\bf Case (1):} $ K_S \equiv  \sum_i a_i F'_i,$
       and if there exists a multiple fibre for which $a_j = m_j -1$, we are done, since then
       $P_n \geq [n/2] + 1$.
       
       Otherwise, there is exactly one multiple fibre, wild, with $t_j=1$,
       and by proposition \ref{Cor} and proposition \ref{Prop2} $a : =a_j$ satisfies
       $$ a = m-1 - \nu =  \nu  ( p^e -1) -1 > 0.$$
       If $\nu = 1$, we obtain $ a/m  = \frac{m-2}{m} \geq 1/3$, if $\nu \geq 2$ then  we get  
       $$ a/m \geq   \frac{p^e - 1 - 1/2}{p^e} =   \frac{2p^e - 3}{2 p^e} \geq 1/4$$
       and accordingly $P_n \geq [n/3] + 1$, $P_n \geq [n/4] + 1$.
       
       \qed
       
              {\bf Case (2):} Again $ K_S \equiv  \sum_i a_i F'_i,$
       and if there exists a multiple fibre for which $a_j = m_j -1$, we are done, since then
       $P_n \geq [n/2] + 1$.
       
       Otherwise there are only wild fibres, either one with $t_1=2$, or two with $t_1, t_2 = 1$.
In the latter case by corollary \ref{KU} we have $a_j = m_j -1 - \nu_j$, and we argue as in case (1).

In the former case we are left (set $m : = m_1, a:= a_1, \nu : = \nu_1$) with the cases 
$a = m-1 -  2 \nu$ or $a = m-1 -  (p+1) \nu$. It is clear that the first  possibility
will give a better estimate than the second, hence we treat the second.

Here $$  \frac{a}{m} =   \frac{p^e - p - 1 - 1/\nu}{p^e}  $$
which   is a monotone increasing function of $e, \nu, p$. 

We must have $ e \geq 2$, and for $ e=2, \nu = 1$ we must have $ p \geq 3$.

In conclusion, for $\nu = 1$, $  \frac{a}{m} \geq min (  \frac{4}{9},   \frac{4}{8}) = \frac{4}{9} \Rightarrow P_n \geq [\frac{4n}{9}] + 1$.

Instead for $\nu \geq 2$, the minimum is for $p=2, e=2, \nu = 2$, and we obtain
$  \frac{a}{m} \geq   \frac{1}{8} $. 

In this case we get  $P_n =  [n/8] + 1$, which would be a limit case, but 
the actual existence of this case with only one multiple fibre, 
and the above numerical characters, is unclear to us.

\qed

{\bf Case (3):} Here $K_S\equiv -F+\sum_{i}^{r}a_iF_i'$, where $F$ is a fibre of $f$. Since some multiple of $K_S$ is linearly equivalent to an effective divisor, we have 
$$(\divideontimes)\ -1+\sum_{i}^{r} \frac{a_i}{m_i}>0,$$
and it follows that $r\geq 2$.
Since $t=1$, there exists one and only one wild fibre, say $m_1F'_1$, with $t_1=1$: by proposition \ref{Cor} and proposition \ref{Prop2} $a_1=m_1-1$, or $a_1=m_1-1-\nu_1$. Hence we can rewrite $K_S$ as follows: $$K_S\equiv -F+a_1F'_1+\sum_{i=2}^{r}(m_i-1)F'_i,$$
so that $$P_n=max(0, 1-n+[\frac{na_1}{m_1}]+\sum_{i=2}^{r}[\frac{n(m_i-1)}{m_i}]).$$

If $r\geq 4$ or $r=3,\ a_1=m_1-1$, we have $P_n\geq 1-n+3[n/2]$, and writing $ n = 2k + s,\ s\in\{0,1\}$,
we get $P_n \geq 1 - 2k -  s  + 3 k = 1 + k -  s $,
which is at least $1$ for $ n \geq 2$, and $\geq 2$ for $n \geq 4$.

In the case where $r=3,\ a_1=m_1-1-\nu_1$, consider first the possibility  $a_1=0$.

 Then ($\divideontimes$) implies that $m_2$ or $m_3$ $\geq 3$, and we get $$P_n\geq 1-n+[2n/3]+[n/2].$$

Writing $n=2k+s$ with $s\in\{0,1\}$, we get
$$P_n\geq 1-2k-s+k+[(k+2s)/3]+k=1+[(k+2s)/3]-s,$$ 
which is $1+[k/3] $ if $s=0$ and $[(k+2)/3] $ when $s=1$. Hence we get $P_n\geq 1$ for $n\geq 2$, $P_6\geq 2$ and $P_n\geq 2$ for $n\geq 8$. 

If instead $a_1>0$, we are of course done if $m_2$ or $m_3$ is $\geq 3$. The remaining case is $m_2=m_3=2$, 
and now  condition $U_1$ implies  that there exists an integer $l$ such that $2(l\nu_1+1)/(p^{e_1}\nu_1)\in\mathbb{Z}$, which implies $\nu_1| 2$. Therefore we conclude that $$\frac{a_1}{m_1}=\frac{m_1-1-\nu_i}{m_1}\geq \frac{m_1-3}{m_1},$$
whence $ m_1 \geq 4$ and  $a_1/m_1\geq 1/4$. Hence we have $$P_n\geq 1-n+[n/4]+2[n/2].$$

We get $P_n\geq 1+[k/2]$ for even $n=2k$ and $P_n\geq [(2k+1)/4]$ for odd $n=2k+1$. Hence we have $P_2\geq 1$, $P_4\geq 2$ and $P_n\geq 2$ for $n\geq 8$. 

We are left with the case $r=2$. 

Assume first that $a_1=m_1-1$: the situation is then identical  to the case $r=3,\ a_1=0$, and we are done.

We may therefore assume that  $a_1=m_1-1-\nu_1>0$, and  inequality $(\divideontimes)$ becomes now 
$$(\divideontimes\divideontimes)\ 1-\frac{1+1/\nu_1}{p^{e_1}}-\frac{1}{m_2}>0 ,$$ 
and we have 
$$P_n\geq 1-n+[\frac{n(p^{e_1}-1-1/\nu_1)}{p^{e_1}}]+[\frac{n(m_2-1)}{m_2}]. $$
 Conditions $U_1,\ U_2$ imply that $\nu_1|m_2$, $m_2 |m_1=p^{e_1}\nu_1$, hence $m_2 = \nu_1 p^{\e}, \e \leq e_1$. 

If $\nu_1=1$, an immediate consequence is that $m_2\geq p$. Moreover, combining with $(\divideontimes\divideontimes)$, we get $p^{e_1}\geq 5$ or $ p^{e_1}= p^{\e}=4$;  but the latter case gives no problems since then   
$$ (***) \ P_n\geq  f_n : = 1-n+[\frac{n}{2}]+[\frac{3n}{4}] = f_s + k , \ n = 4k + s, \ 0 \leq s \leq 3,$$
$$  f_s = 1, 0, 1, 1, \ s= 0,1,2,3.$$
We treat the several cases:
\begin{itemize}
	\item
	If $p\geq 5$, then  $m_2\geq 5$, hence $P_n\geq f_n : = 1-n+[3n/5]+[4n/5]$. Writing $n=5k+s$ with $0\leq s\leq 4$, we get $$P_n\geq f_n  = 2 k + f_s, \ f_ s =  1, 0 , 1, 1, 2, \ s= 0,1,2,3,4.$$

	Therefore we have  $P_n\geq 1$ for $n\geq 2$, and $P_n\geq 2$ for $n\geq 4$. 
	
	\item
	If $p=3$, then  $e_1\geq 2 $ and $m_2\geq 3$. It follows that $P_n\geq f_n : = 1-n+[7n/9]+[2n/3]$.
	 Writing $n=3k+s$ with $0\leq s\leq 2$, we get $$P_n\geq 1-3k-s+2k+[(3k+7s)/9]+2k+[2s/3]$$
	$$=1+k+[(3k+7s)/9]+[2s/3]-s.$$ 
	
	Hence $P_n\geq 1+k$ except for the case $k=0,s=1$, which implies that $P_n\geq 1$ for $n\geq 2$ and $P_n\geq 2$ for $n\geq 3$.
	\item
	If finally $p=2$, observe that $e_1\geq 3$ and $m_2\geq 2$, hence we have $P_n\geq 
	f_n : = 1-n+[3n/4]+[n/2],$ a case which was already treated in $(***)$.
	
\end{itemize}
 Assume now $\nu_1\geq 2$. 
\begin{itemize}
	\item
	If $p^{e_1}\geq 4$, we have that $P_n\geq 1-n+[5n/8]+[n/2]$. Writing $n=2k+s$ with $s\in\{0,1\}$, we get 
	$$P_n\geq1-2k-s+k+[(2k+5s)/8]+k$$
	$$=1+[(2k+5s)/8]-s.$$
	It follows that $P_n\geq 1+[k/4]$ for $s=0$ and $P_n\geq [(2k+5)/8]$ for $s=1$. For the worst case where $p^{e_1}=4$, $\nu_1=m_2=2$ (this case does not actually occur since then condition $U_1$ fails), we have that 
	$P_1=P_3=0$, $P_2=P_4=P_5=P_6=P_7=1$, $P_8=2$ and $P_n\geq 2$ for $n\geq 12$. 
	
	\item
	If $p^{e_1}=3$, we cannot have $m_2=\nu_1=2$, since this would contradict  inequality $(\divideontimes\divideontimes)$. Hence we have either $m_2,\ \nu_1\geq 3$ or $\nu_1=2$, $m_2=6$. We obtain in the respective cases that 
	$$(*1)\ P_n\geq 1-n+[5n/9]+[2n/3]$$
	resp.
	$$(*2)\ P_n\geq 1-n+[n/2]+[5n/6].$$
	
	For $(*1)$, writing $n= 3k+s$ with $0\leq s\leq 2$, we get
	$$P_n\geq 1 +[(6k+5s)/9]+[2s/3]-s,$$
	which implies that $P_n\geq 1+[2k/3]$ for $s=0$, $P_n\geq [(6k+5)/9]$ for $s=1$ and $P_n\geq 1+[(6k+1)/9]$ for $s=2$. Hence $P_n\geq 1$ for $n\geq 2$, $P_6\geq 2$, and $P_n\geq 2$ for $n\geq 8$.
	
	For $(*2)$, writing $n=2k+s$ with $s\in \{0,1\}$, we get 
	$$P_n\geq 1 +[(4k+5s)/6]-s,$$
	it follows that $P_n\geq 1+[2k/3]$ for $s=0$ and $P_n\geq [(4k+5)/6]$ for $s=1$. We see that $P_n\geq 1$ for $n\geq 2$, and $P_n\geq 2$ for $n\geq 4$. 
	\item
	If $p^{e_1}=2$, we have either $m_2=\nu_1,\ \nu_1\geq 4$ or $m_2=2\nu_1,\ \nu_1\geq 3$.
	It follows that 
	$$(*3)\ P_n\geq 1-n+[3n/8]+[3n/4]$$
	resp.
	$$(*4)\ P_n\geq 1-n+[n/3]+[5n/6].$$
	
	For $(*3)$, writing $n=4k+s$ with $0\leq s\leq 3$, we get 
	$$P_n\geq 1+[(4k+3s)/8]+[3s/4]-s,$$
	which equals $1+[k/2]$ for $s=0$, $[(4k+3)/8]$ for $s=1$, $[(4k+6)/8]$ for $s=2$, and $1+[(4k+1)/8]$ for $s=3$. Hence for the worst numerical case $\nu_1=m_2=4$ (this case does not actually occur since again  condition $U_1$ fails), we have $P_3=P_4=P_6=P_7=1$, $P_2=P_5=0$, $P_8=2$, $P_{12}=2$, $P_{13}=1$, and $P_n\geq 2$ for $n\geq 14$. 
	
	For $(*4)$, writing $n=3k+s$ with $0\leq s\leq 2$, we get 
	$$P_n\geq 1+[(3k+5s)/6]-s,$$
	hence $P_n\geq 1+[k/2]$ for $s=0$, $\geq [(3k+5)/6]$ for $s=1$, and $\geq [(3k+4)/6]$ for $s=2$. We conclude that $P_n\geq 1$ for $n\geq 3$, $P_6\geq 2$, and $P_n\geq 2$ for $n\geq 9$.
\end{itemize}

\qed

 {\bf Case (4):} Here  $ K_S \equiv  - 2 F +  \sum_i^r ( m_i -1 )F'_i,$
 since $t=0$ implies that there are no wild fibres.
 
 In view of Theorem \ref{Mon} this situation is exactly as in the classical
 case. But our main theorem is new also in the classical case, so 
 we proceed to treat  case (4).
 
 We assume wlog that
 $$  m_1 \leq m_2 \leq \dots \leq m_r,$$
 and we recall that 
 $$ (****) \ P_n =   max ( 0 , 1 - 2n + \sum_j [ \frac{ n (m_j -1)}{m_j}] ).$$
 
 For $r \geq 5$ we have  $P_n \geq 1 - 2n + 5  [ n/2]$, and writing $ n = 2k + s$,
 we get $P_n \geq 1 - 4k - 2 s  + 5 k = 1 + k - 2 s $,
 which is at least $1$ for $ n \geq 4$, and $\geq 2$ for $n \geq 6$.
 
 Assume $r=4$ and observe once more that the right hand side of $(****)$ 
 is an increasing function of the multiplicities $m_j$, hence the worst case
 is $(2,2,3,3)$. Indeed, the worst case
 would be numerically  $(2,2,2,3)$, but indeed this case does not occur,
 since property $U_4$ is not fulfilled.
 
 Hence the  estimate 
 $$ P_n \geq  1 - 2 n + 2 [n/2] + 2 [2n/3] = 1 + (2 [n/2] - n) + (2 [2n/3] - n). $$
 For even numbers $ n = 2 k$, we get $ P_n =   1 + 2  [k/3] $,
 which is $\geq 1$, and $\geq 3$ as soon as $ n \geq 6$.
 For odd numbers $ n = 2 k + 1$ we get 
 $$ P_n = 2 [\frac{4k+2}{3}] - 2k -1 =  2 [\frac{k+2}{3}] - 1 ,$$
 which is $\geq 1$ for $n \geq 3$,  $\geq 3$ as soon as $ n \geq 9$.
 
 In the case $r=3$ recall that conditions $U_1, U_2, U_3$ are equivalent
 to the condition that $ m_k $ divides the least common multiple of $m_i, m_j$
 for each choice of $ \{ i,j,k \} = \{1,2,3 \}$.
 
 Assume now $ m_1 \geq 4$: by monotonicity the worst case is $(4,4,4)$,
 where, setting $ n = 4k + s, 0 \leq s \leq 3$, 
 $$ P_n \geq    3 [3n/4] - 2n + 1= 3 ( 3k + [3s/4]) - 8 k - 2s + 1 =1 +  k + 3 [3s/4] - 2s. $$
 We get
 \begin{itemize}
 \item 
 $ k+ 1 $ for $s=0, 3$
 \item
 $k$ for $s=2$ 
 \item
 $k-1$ for $s=1$.
 \end{itemize}
 Hence $P_3 = 1, P_4 = 2$, $P_n \geq 2$ for $ n \geq 10$.
 
 Assume now that $m_1 = 3$. Then $3$ divides $LCM(m_2, m_3)$ hence either $m_2 = 3a$ or $3$ does not divide $m_2$ and $m_3 = 3b$.
 or both alternatives hold.
 
Keeping in mind the positivity of $K_S$, equivalent here to  $\sum_j \frac{1}{m_j} < 1$, each alternative leads to a worst possible case,
i.e. one maximizing $\sum_j \frac{1}{m_j} < 1$.

\begin{enumerate}
\item
$(3, 3a, 3b)$ , $a | b$, $b | a \Rightarrow a=b \Rightarrow (3, 3a, 3a)$: worst case $(3,6,6)$;
\item
$(3, c, 3b)$ $c$ not divisible by $3$, $ 3b | 3c, \ c | b \Rightarrow b=c \Rightarrow (3, c, 3c)$:  worst case $(3,4, 12)$;
\item
$(3, 3a, c)$ $c$ not divisible by $3$, $ 3a | 3c, \ c | a \Rightarrow a=c \Rightarrow (3, 3a, a)$: same case as the previous.

\end{enumerate}
Recall the plurigenus formula, here it gives respectively 
 $$ (3,6,6): \  \ P_n =   max ( 0 , F(n)), \ F(n) : = 1 - 2n +  [ \frac{ 2 n }{3}]  + 2   [ \frac{ 5 n }{6}]  .$$
 $$ (3,4,12): \  \ P_n =   max ( 0 , F(n)), \ F(n) : =1 - 2n +  [ \frac{ 2 n }{3}]  +   [ \frac{ 3 n }{4}] +  [ \frac{ 11 n }{12}]  .$$ 
 
 In the former case, writing $ n = 6k + s$, $ 0 \leq s \leq 5$, we get
 $$   \ F (n ) = 2 k +  F(s), \ F(s)  =  1, -1 , 0 , 1, 1, 2    (s=0,1, \dots 5)$$
 hence $P_3 \geq 1, P_5 \geq 2$, $P_n \geq 2 $ for $ n \geq 8$.
 
 In the latter case, writing $ n = 12k + s$, $ 0 \leq s \leq 11$, we get
 $$   \ F (n)  = 4 k  +  F(s),  \ F(0) = 1 , s \geq 1 \Rightarrow  \ F(s) =  -  s +  [ \frac{ 2 s }{3}]  +   [ \frac{ 3 s }{4}] $$
 $$ F(s) = 1,  -1, 0, 1, 1, 1, 2, 2, 3, 3, 3, 4, \ 0 \leq s \leq 11 ,$$
 hence $P_n  \geq 1$ for $ n \geq 3$,  $P_n \geq 2 $ for $ n \geq 6$.
 
 Assume finally $m_1=2$. Then one of $m_2, m_3$ is even. If $m_j = c $ is odd, $m_i = 2b$, then $ c | b, 2b | 2c \Rightarrow b=c$,
 hence we get $(2,b, 2b)$ and the worst case is the case $(2,5,10)$ which was already considered in remark \ref{record}.
 
 Similarly, in case $(2, 2a, 2b)$ again $a=b$, hence we get a triple $(2, 2a, 2a)$ and the worst case is the case $(2,6,6)$ 
 which was already considered in remark \ref{record}.
 
 \begin{remark}\label{limit}
 Our analysis allows us also to see (we omit further details) which are the possible cases where the estimates are sharp in the main theorem.
 
 \begin{itemize}
 \item
 (2):  $P_1 = P_2 = P_3 = 0$ exactly in case (4)  for triples $(2,b,2b), \ b\geq 5, 2 \nmid b,$ or $(2,2a, 2a), \ a \geq 3$.
 \item
(3):  $P_n \leq 1$ for $ n \leq 7$ in case (4) for the triple $(2,5,10)$ and possibly in case (2) with one wild fibre and
$p=\nu=e=2$.
 \item
 (4) : $P_{13} = 1$  exactly in case (4) for the triple $(2,6,6)$.
 
 \end{itemize}
 \end{remark}

   \bigskip

 {\bf Ackowledgement}

Both authors   gratefully acknowledge the support of the 
 ERC Advanced Grant n. 340258, `TADMICAMT'.
 
 The first author would like to thank  Ciro Ciliberto, Mirella Manaresi and Sandro Verra for organizing a Conference in memory
 of Federigo Enriques 70 years after his death: their invitation    provided stimulus to extend Enriques'  $P_{12}$-theorem to positive characteristic;
 and  Igor Dolgachev for     his  question  about item  (3) of the main  theorem
 (which had been treated in \cite{cb}).
  
The second author would like to thank Michel Raynaud for helpful communications.

\end{document}